\documentclass[12pt]{article}
\usepackage{amssymb}
\usepackage{mathrsfs}
\usepackage[hypertex]{hyperref}
\usepackage{amsfonts}
\usepackage{graphicx}
\usepackage{latexsym,amsmath,amssymb,amsfonts,amsthm}

\newcommand{\bcen}{\begin{center}}
\newcommand{\ecen}{\end{center}}
\newtheorem{theorem}{Theorem}[section]
\newtheorem{lemma}[theorem]{Lemma}

\newtheorem{remark}[theorem]{Remark}

\begin{document}
\setcounter{page}{1}
\title{Spectral isoperimetric inequalities for a class of mixed eigenvalue problems of the Laplacian on \\
triangles and trapezoids}
\author{Ruifeng Chen$^{1}$,~~ Jing Mao$^{1,2,\ast}$}

\date{}
\protect \footnotetext{\!\!\!\!\!\!\!\!\!\!\!\!{~$\ast$ Corresponding author} }
\maketitle ~~~\\[-15mm]

\begin{center}
{\footnotesize $^{1}$Faculty of Mathematics and Statistics,\\
 Key Laboratory of Applied
Mathematics of Hubei Province, \\
Hubei University, Wuhan 430062, China\\
 $^{2}$Key Laboratory of Intelligent
Sensing System and Security (Hubei
University), Ministry of Education\\
Emails: gchenruifeng@163.com (R. F. Chen), jiner120@163.com (J. Mao)
 }
\end{center}

%\\[5mm]

\begin{abstract}
In this paper, under suitable geometric constraints, we have
successfully obtained characterizations for the extremum values of
the functional of mixed eigenvalues of the Laplacian on triangles
(or trapezoids) in the Euclidean plane $\mathbb{R}^2$. More
precisely, we prove:
\begin{itemize}

\item For the mixed Dirichlet-Neumann eigenvalue problem of the
Laplacian on triangles with fixed area and the Dirichlet boundary
condition imposed on non-longest two sides (let us call \emph{Type-A
mixed eigenvalue problem}), the isosceles right triangle with the
same area minimizes the first mixed eigenvalue of this eigenvalue
problem. This fact can be seen as a Faber-Krahn type inequality for
the Type-A first mixed eigenvalue of triangles in $\mathbb{R}^2$. As
an application, this Faber-Krahn type isoperimetric inequality gives
a sharp lower bound estimate for the Type-A first mixed eigenvalue
(of triangles) in terms of the area of triangles. Moreover, when
this lower bound is attained, a rigidity result can be obtained as
well.

\item We also give a lower bound estimate (in terms of the length of the longest side
 or the length of the height of the longest side) for the Type-A first mixed eigenvalue of
triangles using another approach.

\item As a byproduct, for the mixed
Dirichlet-Neumann eigenvalue problem of the Laplacian on right
triangles with fixed area and the Dirichlet boundary condition
imposed on the longest side (let us call \emph{Type-B mixed
eigenvalue problem}), we also give a characterization for the
possible extremum value of the first mixed eigenvalue of Type-B.

\item For the mixed Dirichlet-Neumann eigenvalue problem of the
Laplacian on trapezoids in $\mathbb{R}^2$ (with an average width $m$
and a height $h$) having the Dirichlet boundary condition imposed on
the left and the right sides, we show that the rectangle with width
$m$ and height $h$ is a minimizer of the $k$-th mixed eigenvalue
functional, provided $\frac{h}{m}\leq\frac{1}{k}$.

\item For the mixed Dirichlet-Neumann eigenvalue problem of the
Laplacian on right trapezoids in $\mathbb{R}^2$, with fixed area and
having the Dirichlet boundary condition imposed on two parallel
sides and the perpendicular side, the rectangle with the same area
and a length-to-width ratio $2:1$ minimizes the first mixed
eigenvalue of this eigenvalue problem.

\end{itemize}
Besides, we also recalled and proposed several open questions for
readers who are interested in the topic of this paper.
 \end{abstract}

%\markright{\sl\hfill R. F. Chen,~~ J. Mao \hfill}

\vspace{3mm}

\section{Introduction}
\renewcommand{\thesection}{\arabic{section}}
\renewcommand{\theequation}{\thesection.\arabic{equation}}
\setcounter{equation}{0}

The problem of minimizing the first Dirichlet eigenvalue of the
Laplacian on domains with fixed measure in the Euclidean $n$-space
$\mathbb{R}^n$ has a long history, which actually appeared in Lord
Rayleigh's book ``\emph{The theory of sound}"  (for example in the
edition of 1894) more than 100 years ago. It was in the 1920's that
G. Faber \cite{GF} and E. Krahn \cite{EK1} independently and
simultaneously obtained the following classical spectral
isoperimetric inequality (i.e., the Faber-Krahn inequality):

\begin{theorem} \label{theo-1} (\cite{GF,EK1})
Assume that $\Omega$ is an arbitrary open set of $\mathbb{R}^n$, and
let $c$ be a positive number, $\mathbb{B}$ be the Euclidean $n$-ball
of volume $c$. Then\footnote{~In this paper, by the abuse of
notations, $|\cdot|$ also stands for the Hausdorff measure of a
given geometric object. So, here $|\Omega|=c$ means that the volume
of the open domain $\Omega$ is $c$.}
\begin{eqnarray*}
\lambda_1(\mathbb{B})=\min\left\{\lambda_1(\Omega)\Big{|}
|\Omega|=c\right\},
\end{eqnarray*}
where $\lambda_{1}(\cdot)$ stands for the first Dirichlet eigenvalue
of the Laplacian on a prescribed geometric object.
\end{theorem}
\noindent Speaking in other words, from Theorem \ref{theo-1}, one
easily knows:
\begin{itemize}
\item Among all open domains with fixed volume in $\mathbb{R}^n$,
the ball of the same volume minimizes the first Dirichlet eigenvalue
of the Laplacian.
\end{itemize}
We refer to \cite{AH} for a survey on classical results and open
questions about minimization problems concerning the lower
eigenvalues of the Laplace operator. By the way, very recently,
under the constraint of fixed weighted volume and the radial,
concave assumptions for the weighted function, we have successfully
improved this classical Faber-Krahn inequality for the Laplacian to
the case of Witten-Laplacian by mainly using the rearrangement
technique and suitably constructing trial functions (see \cite{CM}).

One can ask a similar problem of minimizing the first Dirichlet
eigenvalue of the Laplacian on polygons with $n$ sides and fixed
area in $\mathbb{R}^2$. G. P\'{o}lya has given a nice partial answer
to this question. In fact, he has proven:
\begin{theorem} (G. P\'{o}lya, see e.g. \cite{AH,AH1}) \label{theo-2}
The equilateral triangle has the least eigenvalue among all
triangles of given area. The square has the least first eigenvalue
among all quadrilaterals of  given area.
\end{theorem}
 \noindent The main tool in P\'{o}lya's proof of Theorem \ref{theo-2} is
the so-called Steiner symmetrization (see e.g. \cite{PS} for this
notion). Curiously, comparing with the case of triangles, the proof
for the case of quadrilaterals is a little bit
 simpler. Indeed, a sequence of
three Steiner symmetrizations allows us to transform any
quadrilateral into a rectangle.  Therefore, it is sufficient to look
at the minimization problem for the first Dirichlet eigenvalue of
the Laplacian among rectangles. Since the first Dirichlet eigenvalue
$\lambda_1(\Omega)$ on any rectangular domain $\Omega$ can be
calculated directly, it is easy by using the Cauchy-Schwarz
inequality to see that a square with the same area $|\Omega|$ can
minimize the functional $\Omega\mapsto\lambda_1(\Omega)$. However,
P\'{o}lya's this idea cannot work any more for polygons with $n$
sides ($n\geq5$) since in general the Steiner symmetrization
increases the number of sides of $n$-gones (pentagons and others).
So, a natural and hard question is:
 \begin{itemize}
\item \textbf{Open Problem 1}. (see \cite[Subsection 3.2]{AH}) Prove that the regular $n$-gone has the least first Dirichlet
eigenvalue of the Laplacian among all the $n$-gones of given area
for $n\geq5$.
 \end{itemize}
As mentioned by A. Henrot at the end of \cite[Subsection 3.2]{AH},
the classical isoperimetric inequality linking area and length for
regular $n$-gones (see e.g. \cite[Theorem 5.1]{RO}) supports the
above conjecture and lets us have the reason that \textbf{Open
Problem 1} might be true. By the way, the \textbf{Open Problem 1} is
known as P\'{o}lya-Szeg\"{o} conjecture, which was formally raised
by G. P\'{o}lya and G. Szeg\"{o} in \cite[page 158]{PS}. Actually,
except the isoperimetric inequality linking area and length for
regular $n$-gones,  many numerical experiments have been performed
for small values of $n$, also suggesting the validity of the
P\'{o}lya-Szeg\"{o} conjecture.

Laplacian eigenvalues are usually interpreted as \emph{frequencies}
of vibrating membranes. Actually, in the context of eigenvalue
problems of the Laplacian, Dirichlet boundary condition (BC)
corresponds to the vibration of a membrane with boundary fixed,
while Neumann BC indicates a free membrane. Based on this reason,
intuitively, (as said in \cite{SB}) a mixed Dirichlet-Neumann BC
means that the membrane is partially attached. Besides, because of
the relation between Dirichlet eigenvalues (of the Laplacian) and
Neumann eigenvalues, generally it follows that the larger the
attached portion, the higher the frequencies.

The above facts invoke us to consider mixed eigenvalue problems of
the Laplacian on planar $n$-gones with $n\geq3$, i.e. the
eigenequaiton of the Laplacian holds in the inner part of $n$-gones,
while a Dirichlet-Neumann BC was imposed on the $n$-sides. The mixed
Dirichlet-Neumann conditions (in literatures) are also known as the
Zaremba problem (see \cite{ZS}), and can be highly singular, in
particular when a transition between the Dirichlet and Neumann
conditions happens in the middle of a side of a polygon (see e.g.
\cite{BR,OB} for details). Based on this reason, in this paper, we
investigate a mixed eigenvalue problem of the Laplacian, with
imposing Dirichlet conditions on various combinations of the sides
of a triangle or a trapezoid, and only treat the cases with
transitions at vertices, which pose no problems.

Let $T\subset\mathbb{R}^2$ be a planar triangle in the Euclidean
plane $\mathbb{R}^2$, and let $L$, $M$, $S$ be the sides of the
triangle $T$ such that their lengths satisfy $\mathcal{L}(L)\geq
\mathcal{L}(M)\geq\mathcal{L}(S)$. Denote by $\lambda_{1}^{set}$ the
smallest eigenvalue corresponding to the Dirichlet conditions
applied to a chosen set of sides -- in this setting, e.g.,
$\lambda_{1}^{LS}$ would correspond to the Dirichlet conditions
imposed on the longest and shortest sides. Consider the mixed
Dirichlet-Neumann eigenvalue problem of the Laplacian on the
triangle $T$ as follows
\begin{equation} \label{mixed-eigenvalue}
\left\{ \begin{array}{lll}
\Delta u + \lambda^{\mathcal{D}} u=0  &\mathrm{in}~~T,\\
u=0 &\mathrm{on}~~\mathcal{D}\subset \{L,M,S\},\\
\partial_{\nu}u=0 &\mathrm{on}~~\partial T\setminus \mathcal{D},
\end{array} \right.
\end{equation}
where $\mathcal{D}$ can be any combination of the triangle's sides,
and $\nu$ denotes the unit outward normal vector along the set
$\partial T\setminus \mathcal{D}$. As usual, $\Delta$ denotes the
Laplace operator. For simplicity and convenience, denote by
$\lambda=\lambda^{LMS}$ the (pure) Dirichlet eigenvalues and
$\mu=\lambda^{\emptyset}$ the (pure) Neumann eigenvalues. It is
well-known that the mixed eigenvalue problem
(\ref{mixed-eigenvalue}) only has a discrete spectrum, since $T$ is
bounded. Moreover, from \cite[Section 2]{SB}, it is not hard to see
that eigenvalues in this discrete spectrum can be listed
non-decreasingly as follows
\begin{eqnarray*}
0<\lambda^{\mathcal{D}}_{1}(T)<\lambda^{\mathcal{D}}_{2}(T)\leq\lambda^{\mathcal{D}}_{3}(T)\leq\cdots\uparrow\infty,
\end{eqnarray*}
as long as $\mathcal{D}$ is nonempty. When $\mathcal{D}$ is empty
(i.e. the purely Neumann case), one has\footnote{~Without
specification and for convenience, in the sequel we wish to write
the $i$-th mixed eigenvalue $\lambda^{\mathcal{D}}_{i}(T)$ as
$\lambda^{\mathcal{D}}_{i}$, $i\in\mathbb{Z}_{+}$ with
$\mathbb{Z}_{+}$ the set of all positive integers. This convention
would be used similarly for other kinds of eigenvalues investigated
in this paper. For instance, the $i$-th Neumann eigenvalue
$\mu_{i}(T)$ would be simply written as $\mu_{i}$.}
 \begin{eqnarray*}
0=\mu_{1}(T)<\mu_{2}(T)\leq\mu_{3}(T)\leq\mu_{4}(T)\leq\cdots\uparrow\infty.
 \end{eqnarray*}
Besides, all the eigenvalues of (\ref{mixed-eigenvalue}) can be
obtained by minimizing the following Rayleigh quotient
\begin{eqnarray*}
\mathcal{R}[u]=\frac{\int_{T}|\nabla u|^2}{\int_{T}u^{2}},
\end{eqnarray*}
where $\nabla$ is the gradient operator in $\mathbb{R}^2$. Here, for
convenience, we have dropped the integral measure in all the
integrations of $\mathcal{R}[u]$, and, without specification, this
convention would
 be also used in the sequel. Specially, the first mixed eigenvalue
 $\lambda^{\mathcal{D}}_{1}(T)$ and the first nonzero Neumann
 eigenvalue $\mu_{2}(T)$ can be characterized as follows:
 \begin{eqnarray*}
\lambda^{\mathcal{D}}_{1}(T)=\inf\left\{\mathcal{R}[u]\bigg{|}u\in
W^{1,2}(T),u=0~\mathrm{on}~\mathcal{D}\right\}
 \end{eqnarray*}
and
 \begin{eqnarray*}
\mu_{2}(T)=\inf\left\{\mathcal{R}[u]\bigg{|}u\in
W^{1,2}(T),\int_{T}u=0\right\},
 \end{eqnarray*}
where as usual $W^{1,2}(T)$ stands for a Sobolev space (i.e. the
completion of the set of smooth functions $C^{\infty}(T)$ under the
Sobolev norm $\|u\|_{1,2}:=\int_{T}u^{2}+\int_{T}|\nabla
u|^{2}<\infty$). We refer readers to e.g. \cite{BC} for an overview
of the variational approach.

Nearly ten years ago, B. Siudeja \cite{SB} considered the mixed
eigenvalue problem (\ref{mixed-eigenvalue}) and gave an order for
the mixed Dirichlet-Neumann eigenvalues. In fact, he proved:

\begin{theorem} (\cite{SB}) \label{theo-3}
For any right triangle with the smallest angle satisfying
$\pi/6<\alpha<\pi/4$, it holds
\begin{equation*}
0=\mu_{1}<\lambda_{1}^{S}<\lambda_{1}^{M}<\mu_{2}<\lambda_{1}^{L}<\lambda_{1}^{MS}<\lambda_{1}^{LS}
<\lambda_{1}^{LM}<\lambda_{1}.
\end{equation*}
When $\alpha=\pi/6$ (half-of-equilateral triangle) we have
$\lambda_1^M=\mu_2$, and for $\alpha=\pi/4$ (right isosceles
triangle) we have $S=M$ and $\lambda_{1}^{L} =\mu_{2}$. All other
inequalities stay sharp in these cases.

Furthermore for arbitrary triangle, it holds
\begin{equation*}
\min\{\lambda_{1}^{S}, \lambda_{1}^{M},
\lambda_{1}^{L}\}<\mu_{2}<\lambda_{1}^{MS}<\lambda_{1}^{LS}<\lambda_{1}^{LM}
\end{equation*}
as long as the appropriate sides have different lengths. However, it
is possible that $\mu_{2}>\lambda_{1}^{L}$ (for any small
perturbation of the equilateral triangle) or $\mu_{2}<\lambda_{1}^M$
(for right triangles with $\alpha<\pi/6$).
\end{theorem}

\noindent Moreover, he conjectured that all cases missing in Theorem
\ref{theo-3} are still true:
 \begin{itemize}
\item \textbf{Open Problem 2}. (see \cite[Section 1]{SB}) For an arbitrary
triangle, it holds
$$\lambda_{1}^{S}<\lambda_{1}^{M}<\lambda_{1}^{L}<\lambda_{1}^{MS}$$
as long as appropriate sides have different lengths.
 \end{itemize}
From Theorem \ref{theo-2}, one knows that for the extremum problem
\begin{eqnarray*}
\min\left\{\lambda_{1}^{LMS}(T)=\lambda_{1}(T)\Big{|} |T|=c\right\},
\end{eqnarray*}
equilateral triangles of the same area $c$ should be the unique
minimizer (up to the congruence). Obviously, in this case,
$\mathcal{D}=\{L,M,S\}$. A natural and interesting problem is trying
to solve
\begin{eqnarray} \label{1-2}
\min\left\{\lambda_{1}^{\mathcal{D}}(T)\Big{|} |T|=c\right\},
\end{eqnarray}
if $\mathcal{D}$ is not equal to the set $\{L,M,S\}$ and is not
empty.

We have successfully solved the extremum problem (\ref{1-2}) in the
case $\mathcal{D}=\{M,S\}$. More precisely, we have proven:

\begin{theorem} \label{theo-4}
Assume that $T$ is an arbitrary triangle in $\mathbb{R}^2$, and
$T^{\ast}$ is an isosceles right triangle with the same area as $T$,
i.e. $|T|=|T^{\ast}|$. Then one has
\begin{eqnarray} \label{1-3}
\lambda_{1}^{MS}(T)\geq\lambda_{1}^{MS}(T^{\ast}).
\end{eqnarray}
Moreover, the equality in (\ref{1-3}) can be attained if and only if
$T$ is congruent with $T^{\ast}$.
\end{theorem}

\noindent Speaking in other words, from Theorem \ref{theo-4}, one
knows:
\begin{itemize}
\item Among all triangles with fixed area in $\mathbb{R}^{2}$, the isosceles right
triangle of the same area minimizes the functional
$T\mapsto\lambda_{1}^{MS}(T)$.
\end{itemize}
As a byproduct, we can also prove:
\begin{theorem}  \label{theo-5}
Let $T_{R}$ be an arbitrary right triangle in $\mathbb{R}^2$, and
$T^{\ast}_{R}$ be an isosceles right triangle with the same area as
$T_{R}$, i.e. $|T_{R}|=|T^{\ast}_{R}|$. Then
$\lambda_{1}^{L}(T_{R})\geq\lambda_{1}^{L}(T^{\ast}_{R})$.
\end{theorem}

\begin{remark} \label{remark-1}
\rm{ (1) At the place which is right behind the proof of Theorem
\ref{theo-4} given in Section \ref{sect2}, we have shown that
$\lambda_{1}^{MS}(T^{\ast})=\pi^{2}/|T^{\ast}|$. Therefore, the
statement of Theorem \ref{theo-4} can be rewritten as:
\begin{itemize}
\item \emph{For an arbitrary triangle $T$ in $\mathbb{R}^2$, one has
 \begin{eqnarray}  \label{1-4}
\lambda_{1}^{MS}(T)\geq\frac{\pi^{2}}{|T|},
 \end{eqnarray}
 and the equality can be attained if and only if $T$ is an isosceles right triangle.}
\end{itemize}
(2) Clearly, the assertion of Theorem \ref{theo-2} can be seen as a
Faber-Krahn type inequality for the first mixed eigenvalue
$\lambda^{MS}_{1}(\cdot)$ of triangles in $\mathbb{R}^2$. Besides,
our conclusion here can also be seen as a complement to the lower
bound estimate for the first eigenvalue of the mixed eigenvalue
 problem (2.1) (considered in \cite[Sections 2 and 3]{DD}) of the
Laplacian on a bounded $C^2$-domain in $\mathbb{R}^n$, $n\geq2$.
 }
\end{remark}

The meaning of Remark \ref{remark-1} is:
\begin{itemize}
\item By applying Theorem \ref{theo-2}, a lower bound for the
first mixed eigenvalue $\lambda^{MS}_{1}(T)$ in terms of the area of
a given triangle $T$ can be successfully given. Moreover, an
intersting rigidity conclusion can be obtained when the lower bound
was achieved.
 \end{itemize}
 Is it possible to give interesting lower bounds for
 $\lambda^{MS}_{1}(T)$ in terms of other important geometric quantities of the
  triangle $T$? The answer is affirmative. In fact, motivated by P.
  Freitas's work \cite{FP} of giving upper and lower bounds for the
  first Dirichlet eigenvalue of a triangle, we can prove the
  following lower bound estimate.

\begin{theorem} \label{theo-6}
For an arbitrary triangle $T$ in $\mathbb{R}^2$, we have
 \begin{eqnarray}  \label{1-5}
\lambda_{1}^{MS}(T)\geq\min\left\{\frac{\pi^2}{h^2},\frac{4\pi^2}{\ell^2}\right\},
 \end{eqnarray}
where $\ell$ is the length of the longest side $L$ of the triangle
$T$, and $h$ is the length of the height of the side $L$.
\end{theorem}

\begin{remark}
\rm{ Let $\alpha$ be the angle corresponding to the longest side
$L$, and here let $s$, $m$, $\ell$ be the lengths of sides $S$, $M$
and $L$ of the triangle $T$, respectively. We divide our argument
into two cases:

Case I. When $\alpha\geq\frac{\pi}{2}$ (i.e. $T$ is a right triangle
or an obtuse triangle), one has
\begin{eqnarray*}
m^{2}+s^{2}\geq2ms\geq2ms\sqrt{2}\sin(\alpha+\frac{\pi}{4}),
\end{eqnarray*}
which implies
\begin{eqnarray*}
m^{2}+s^{2}\geq2ms(\sin\alpha+\cos\alpha).
\end{eqnarray*}
Simplifying the above inequality yields
\begin{eqnarray*}
m^{2}+s^{2}-2ms\cos\alpha\geq2ms\sin\alpha,
\end{eqnarray*}
which directly implies
 \begin{eqnarray*}
\ell^{2}\geq4|T|.
 \end{eqnarray*}
Especially, if $\ell^{2}=4|T|$, then $m=s$ and $\alpha=\pi/2$, which
implies that in this setting $T$ should be an isosceles right
triangle.

Case II. Since $\alpha$ is the angle corresponding to the longest
side $L$, it is not hard to get $\alpha\geq\frac{\pi}{3}$. Now, we
discuss the case $\frac{\pi}{3}\leq\alpha<\frac{\pi}{2}$ (i.e. $T$
is an acute triangle).

(1) If furthermore $m=s$ (i.e. $T$ is an isosceles acute triangle),
then we first have $\sin(\alpha+\frac{\pi}{4})>\frac{\sqrt{2}}{2}$,
which implies
\begin{eqnarray*}
2s^{2}<2s^{2}\sqrt{2}\sin(\alpha+\frac{\pi}{4}),
\end{eqnarray*}
and so
\begin{eqnarray*}
2s^{2}<2s^{2}(\sin\alpha+\cos\alpha).
\end{eqnarray*}
Simplifying the above inequality yields
\begin{eqnarray*}
2s^{2}-2s^{2}\cos\alpha<4\times\left(\frac{1}{2}s^{2}\sin\alpha\right),
\end{eqnarray*}
that is, $\ell^{2}<4|T|$.

(2) If furthermore $m\neq s$, then the relation between $\ell^2$ and
$4|T|$ is uncertain and deeply depends on the choice of values of
$m$, $s$. We wish to give two examples to explain this situation.
For instance, considering triangles  with the maximal interior angle
$\alpha=5\pi/12$, if $m=4$, $s=3$, a simple calculation gives us
\begin{eqnarray*}
4|T|\approx23.18, \qquad \ell^{2}\approx18.7888,
\end{eqnarray*}
which means $4|T|>\ell^2$ in this setting. However, if $\ell=4$,
 $m=3.9823$, $s=2$, a simple calculation tells us
\begin{eqnarray*}
 4|T|\approx15.4548, \qquad \ell^{2}=16,
\end{eqnarray*}
which implies $4|T|<\ell^2$ in this setting.

In sum, we have:
\begin{itemize}
\item (a) If $T$ is an obtuse triangle, a right triangle (except the case of
isosceles right triangles), or an acute triangle which has the form
like the triangle $\ell=4$, $m=3.9823$, $s=2$ constructed in (2) of
Case II, then $\ell^{2}>4|T|$.

\item (b) If $T$ is an isosceles right triangle, then
$\ell^{2}=4|T|$.

\item (c) If $T$ is an isosceles acute triangle or an acute triangle which has the form
like the triangle $m=4$, $s=3$, $\alpha=5\pi/12$ constructed in (2)
of Case II, then $\ell^{2}<4|T|$.
\end{itemize}
On the other hand, it is not hard to see that:
\begin{itemize}

\item If $0<h<\frac{\ell}{2}$, then $h^{2}<\frac{\ell^{2}}{4}$ and
$|T|<\frac{\ell^{2}}{4}$. In this situation, $T$ must be a triangle
belonging to the category (a) described as above, and moreover, the
estimate (\ref{1-5}) in Theorem \ref{theo-6} becomes
$\lambda_{1}^{MS}(T)\geq\frac{4\pi^2}{\ell^2}$. Since
$|T|=\frac{h\ell}{2}<\frac{\ell^{2}}{4}$, one knows that in this
situation, the estimate (\ref{1-4}) is sharper than the estimate
(\ref{1-5}).

\item If $h=\frac{\ell}{2}$, then $h^{2}=\frac{\ell^{2}}{4}$ and
$|T|=\frac{\ell^{2}}{4}$. In this situation, $T$ must be an
isosceles right triangle, i.e. belonging to the category (b).
Moreover, the estimate (\ref{1-5}) becomes
$\lambda_{1}^{MS}(T)\geq\frac{\pi^2}{|T|}$, which coincides with the
lower bound estimate given in (\ref{1-4}).

\item If $h>\frac{\ell}{2}$, then $h^{2}>\frac{\ell^{2}}{4}$ and
$|T|>\frac{\ell^{2}}{4}$. In this situation, $T$ must be a triangle
belonging to the category (c), and moreover, the estimate
(\ref{1-5}) becomes $\lambda_{1}^{MS}(T)\geq\frac{\pi^2}{h^2}$.
Since $|T|=\frac{h\ell}{2}<h^{2}$, one knows that in this situation,
the estimate (\ref{1-4}) is sharper than the estimate (\ref{1-5}).

\end{itemize}
Therefore, if $T$ is an isosceles right triangle, then the lower
bound in the estimate (\ref{1-4}) or (\ref{1-5}) coincides with each
other, while if $T$ is an obtuse triangle, a right triangle (except
the case of isosceles right triangles), or an acute triangle, then
the lower bound estimate (\ref{1-4}) is strictly sharper than the
estimate (\ref{1-5}).
 }
\end{remark}

Given a trapezoid in $\mathbb{R}^2$, assume that the mean length of
the two parallel sides of this trapezoid is $m$ (we call $m$ the
mean width or the average width of the trapezoid), and the height is
$h$. For Dirichlet eigenvalues of the Laplacian on this given
trapezoid, G. P\'{o}lya \cite{PG} proved:

\begin{theorem} (\cite{PG}) \label{theo-7}
When $\frac{h}{m}\leq \frac{\sqrt 3}{k}$, $k=1,2,3,\cdots$. A
rectangle with an average width $m$ and a height $h$ achieves the
minimum value of the $k$-th Dirichlet eigenvalue $\lambda_{k}$.
\end{theorem}

Inspired by the conclusion of Theorem \ref{theo-7}, it is natural to
consider a similar question for the mixed Dirichlet-Neumann
eigenvalue problem of the Laplacian on trapezoids.

Let $\Gamma$ be a trapezoid in $\mathbb{R}^2$. Denote the shorter
 and the longer parallel sides of $\Gamma$ by $P_1$, $P_2$
respectively. Let the rest sides (i.e. the left and the right legs
of $\Gamma$) be denoted by $Q_1$, $Q_2$. Consider the mixed
Dirichlet-Neumann eigenvalue problem of the Laplacian on the
trapezoid $\Gamma$ as follows
\begin{equation} \label{mixed-eigenvalue-1}
\left\{ \begin{array}{lll}
\Delta u + \lambda^{\mathcal{D}} u=0  &\mathrm{in}~~\Gamma,\\
u=0 &\mathrm{on}~~\mathcal{D}\subset \{P_1,P_2,Q_1,Q_2\},\\
\partial_{\nu}u=0 &\mathrm{on}~~\partial \Gamma\setminus \mathcal{D},
\end{array} \right.
\end{equation}
where, similar as before, $\mathcal{D}$ can be any combination of
the trapezoid's sides, and $\nu$ denotes the unit outward normal
vector along the set $\partial\Gamma\setminus \mathcal{D}$. Of
course, as explained before, in order to avoid singularities arisen,
 we only treat the cases with transitions at vertices of $\Gamma$.
 Similar as the case of triangles, it is not hard to see that
the mixed eigenvalue problem (\ref{mixed-eigenvalue-1}) only has a
discrete spectrum, and eigenvalues in this discrete spectrum can be
listed non-decreasingly as follows
\begin{eqnarray*}
0<\lambda^{\mathcal{D}}_{1}(\Gamma)<\lambda^{\mathcal{D}}_{2}(\Gamma)\leq\lambda^{\mathcal{D}}_{3}(\Gamma)\leq\cdots\uparrow\infty,
\end{eqnarray*}
as long as $\mathcal{D}$ is nonempty. When $\mathcal{D}$ is empty
(i.e. the purely Neumann case), one has
 \begin{eqnarray*}
0=\mu_{1}(\Gamma)<\mu_{2}(\Gamma)\leq\mu_{3}(\Gamma)\leq\mu_{4}(\Gamma)\leq\cdots\uparrow\infty.
 \end{eqnarray*}
By applying the variational principle, one knows that all the
eigenvalues of (\ref{mixed-eigenvalue-1}) can be obtained by
minimizing the Rayleigh quotient $\mathcal{R}[u]$ (with $T$ replaced
by $\Gamma$), and the first mixed eigenvalue
 $\lambda^{\mathcal{D}}_{1}(\Gamma)$, the first nonzero Neumann
 eigenvalue $\mu_{2}(\Gamma)$ can be characterized similarly. Using
 the convention before, if $\mathcal{D}=\{P_1,P_2,Q_1,Q_2\}$, we
 wish to write
 $\lambda_{k}^{\mathcal{D}}(\Gamma)=\lambda_{k}^{P_{1}P_{2}Q_{1}Q_{2}}(\Gamma)$ as
 $\lambda_{k}(\Gamma)$, $k=1,2,\cdots$. Besides, if there is no
 confusion, $\lambda_{k}(\Gamma)$ can be simplified as $\lambda_{k}$
 directly.

For the mixed eigenvalue problem (\ref{mixed-eigenvalue-1}), we
have:
\begin{theorem} \label{theo-8}
Assume that $\Gamma\subset\mathbb{R}^2$ is a trapezoid with an
average width $m$ and a height $h$, and assume that $\Gamma^{\ast}$
is a rectangle with width $m$ and height $h$. If
$\frac{h}{m}\leq\frac{1}{k}$, then one has
 \begin{eqnarray*}
 \lambda_{k}^{Q_{1}Q_{2}}(\Gamma)\geq \lambda_{k}^{Q_{1}Q_{2}}(\Gamma^\ast)
 \end{eqnarray*}
 for $k\in\mathbb{Z}_{+}$.
\end{theorem}

We also have:
\begin{theorem} \label{theo-9}
Assume that $\Gamma_{R}$ is a right trapezoid in $\mathbb{R}^2$.
Denote the  upper  and lower parallel sides of $\Gamma_{R}$ by
$l_1$, $l_2$ respectively. Let $w_2$ be the side of $\Gamma_{R}$,
which is perpendicular with parallel sides $l_1$, $l_2$, and let
$w_1$ be the last slant side. Assume that $\Gamma_{R}^{\ast}$ is a
rectangle with area $|\Gamma_{R}^{\ast}|=|\Gamma_{R}|$ (i.e.
$\Gamma_{R}^{\ast}$'s area equals the area $|\Gamma_{R}|$ of
$\Gamma_{R}$) and a length-to-width ratio $2:1$. Then we have
 \begin{eqnarray*}
 \lambda_{1}^{l_{1}l_{2}w_{1}}(\Gamma_{R})\geq\lambda_{1}^{l_{1}l_{2}w_{1}}(\Gamma_{R}^{\ast}).
 \end{eqnarray*}
 Moreover, the equality can be achieved if and only if $\Gamma_{R}$
 is congruent with $\Gamma_{R}^{\ast}$.
\end{theorem}

The paper is organized as follows. In Section \ref{sect2}, the
proofs of the conclusions for triangles will be given. The case of
trapezoids will be dealt with in Section \ref{sect3}.

\section{The case of triangles}
\renewcommand{\thesection}{\arabic{section}}
\renewcommand{\theequation}{\thesection.\arabic{equation}}
\setcounter{equation}{0}  \label{sect2}

We give the proof of our first main result.

\begin{proof} [Proof of Theorem \ref{theo-4}]
For an arbitrary triangle $T\subset\mathbb{R}^2$, we use red solid
line segments to represent the sides assigned the Dirichlet boundary
condition, and the black dashed line segment to represent the side
assigned the Neumann boundary condition. Assume that $u$ is an
eigenfunction corresponding to $\lambda_{1}^{MS}(T)$. Now, we fold
the triangle $T$ along the longest line $L$, and denote the
resulting quadrilateral by $Q$. One knows that the area $|Q|$ of $Q$
should be twice of $|T|$, as shown in the Figure \ref{f1}.
\begin{figure}[ht]
\centering
 \includegraphics[width=0.35\textwidth]{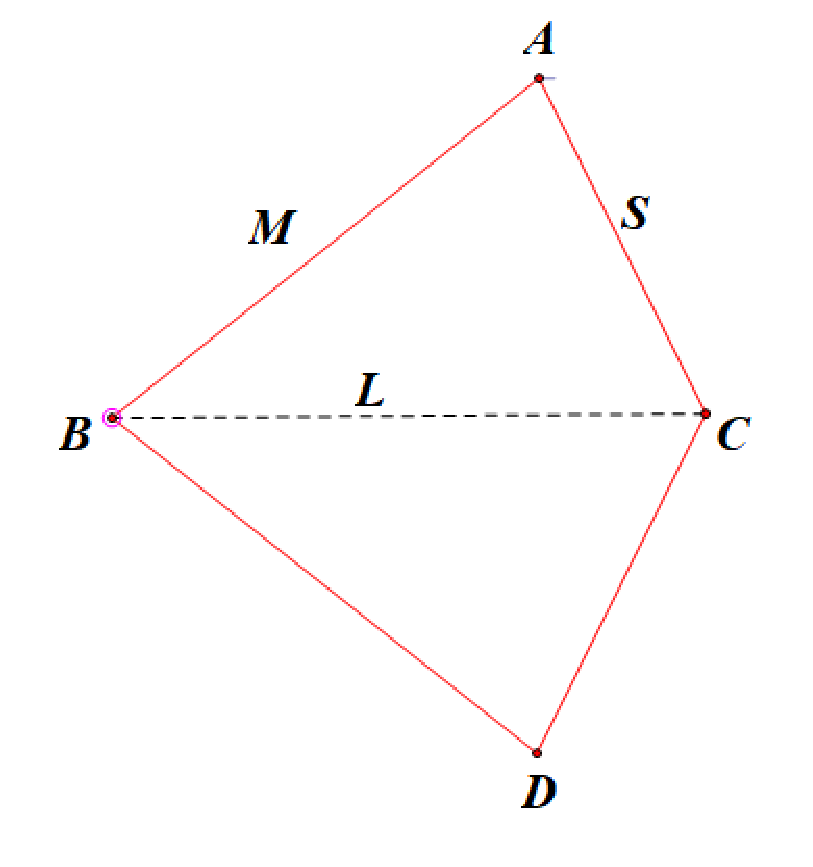}
\caption{The quadrilateral $Q$ generated by the triangle $T$ through
symmetrization} \label{f1}
\end{figure}

Construct a function $f$ defined on $Q$ as follows
\begin{eqnarray*}
f(x)=\left\{
\begin{array}{ll}
u(x) ~~&\mathrm{if}~x\in T, \\
u(y) ~~&\mathrm{if}~x\in Q\setminus T,
\end{array}
 \right.
\end{eqnarray*}
where $y$ is the symmetric point of $x$ with respect to the side
$BC$.  Due to the symmetry of $f$, one has
\begin{eqnarray*}
\int_T |\nabla u|^2 =\int_{Q\setminus T} |\nabla f|^2\\
 \end{eqnarray*}
 and
\begin{eqnarray*}
\int_T | u|^2 =\int_{Q\setminus T}|f|^{2}.
\end{eqnarray*}
Since $f$ can be regarded as a trial function for the first
Dirichlet eigenvalue of the Laplacian on $Q$, one can get
\begin{eqnarray}\label{mix0}
\lambda_{1}^{MS}(T)&=& \frac{\int_T|\nabla u|^2}{\int_T
u^2}=\frac{\int_T|\nabla u|^2 +
\int_{Q\setminus T}|\nabla f|^2}{\int_{T} u^2 + \int_{Q\setminus T} f^2}\nonumber\\
&=&\frac{\int_Q|\nabla f|^2}{\int_Q f^2}\geq\lambda_{1}(Q).
\end{eqnarray}
By (\ref{mix0}) and Theorem \ref{theo-2} , one has
\begin{eqnarray}\label{mix1}
\lambda_{1}^{MS}(T)\geq \lambda_{1}(Q)\geq \lambda_{1}(S_{Q})
\end{eqnarray}
where $S_{Q}\subset\mathbb{R}^2$ is a square having the same area as
$Q$, i.e. $|S_{Q}|=|Q|$. Since an eigenfunction $h$ of the first
Dirichlet eigenvalue of the Laplacian on a square is symmetric with
respect to the diagonal of this square, one can divide the above
square $S_{Q}$ along its diagonal into two parts of equal area (i.e.
two congruent isosceles right triangles). Let one of these isosceles
right triangles be as $T^{\ast}$. Denote by the function $h$
restricted on $T^{\ast}$ as $g$, and then it is clear that $g$ takes
the value $0$ on the two perpendicular sides of triangle $T^{\ast}$.
By the variational principle of the mixed Dirichlet-Neumann
eigenvalue problem (\ref{mixed-eigenvalue}), it is true that the
function $g$ can be served as a trial function for the eigenvalue
$\lambda_{1}^{MS}(T^{\ast})$. Due to the symmetry of $h$ with
respect to the diagonal of $S_{Q}$, it follows that
\begin{eqnarray} \label{2-3}
\lambda_1(S_{Q})&=& \frac{\int_{S_{Q}}|\nabla h|^2}{\int_{S_{Q}}
h^2}=\frac{\int_{T^{\ast}}|\nabla g|^2 +
\int_{S_{Q}\setminus T^{\ast}}|\nabla h|^2}{\int_{T^{\ast}} g^2 + \int_{S_{Q}\setminus T^{\ast}}h^2}\nonumber\\
&=&\frac{\int_{T^{\ast}}|\nabla g|^2}{\int_{T^{\ast}}
g^2}\geq\lambda_{1}^{MS}(T^{\ast}).
\end{eqnarray}
Together with (\ref{mix1}), it yields $\lambda_{1}^{MS}(T)\geq
\lambda_{1}^{MS}(T^{\ast})$. Moreover, when $\lambda_{1}^{MS}(T)=
\lambda_{1}^{MS}(T^{\ast})$, one directly gets
$\lambda_{1}(Q)=\lambda_{1}(S_{Q})$, which by Theorem \ref{theo-2}
implies that $Q$ is congruent with the square $S_{Q}$. Since $Q$ is
generated by folding the triangle $T$ along its side $L$, one knows
 in this situation that $T$ must be congruent with $T^{\ast}$. This
completes the proof of Theorem \ref{theo-4}.
\end{proof}

Using a similar argument to the derivation of
(\ref{mix0})-(\ref{mix1}), it is not hard to see that
\begin{eqnarray*}
\lambda^{MS}_{1}(T^{\ast})\geq\lambda_{1}(S_{Q}).
\end{eqnarray*}
Together with (\ref{2-3}), one has
\begin{eqnarray*}
\lambda^{MS}_{1}(T^{\ast})=\lambda_{1}(S_{Q})=\frac{2\pi^{2}}{|S_{Q}|}=\frac{\pi^{2}}{|T^{\ast}|}.
\end{eqnarray*}

The proof of the second main result can be given as follows:

\begin{proof} [Proof of Theorem \ref{theo-5}]
We take the right triangle $T_{R}:=\triangle ABC$ and reflect it
along its two perpendicular sides to form a rhombus $\mathfrak{R}$,
as shown in Figure \ref{f2} below.
\begin{figure}[ht]
\centering
 \includegraphics[width=0.35\textwidth]{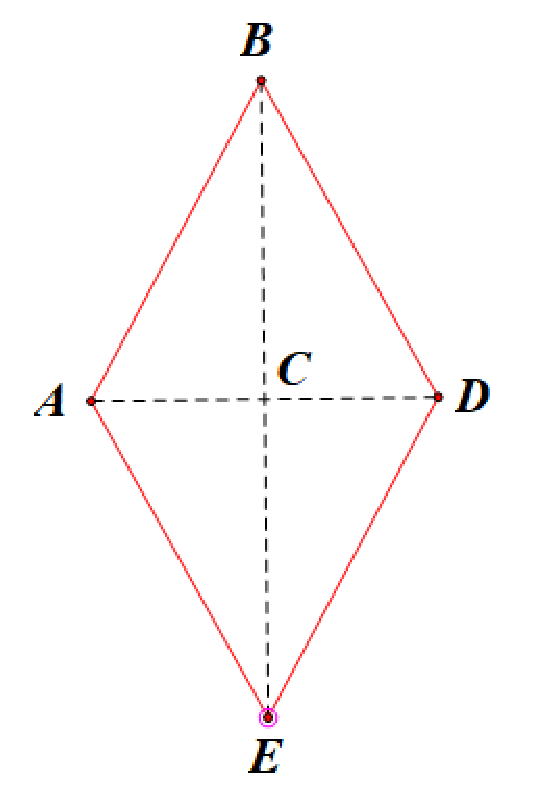}
\caption{The rhombus $\mathfrak{R}$ formed from the right triangle
$\triangle ABC$ by reflection} \label{f2}
\end{figure}

Let $v$ be the characteristic function corresponding to
$\lambda_{1}^{L}(T_{R})$. Construct a function $\phi$ as follows:
\begin{eqnarray*}
\phi(x)= \left\{ \begin{array}{llll}
 v(x) \qquad &\mathrm{if}~x\in T_{R}, \\
 v(y) \qquad &\mathrm{if}~x\in \triangle BCD,\\
v(z) \qquad &\mathrm{if}~x\in \triangle ACE,\\
 v(w) \qquad &\mathrm{if}~x\in \triangle DCE,
\end{array} \right.
\end{eqnarray*}
where $y$ is the symmetric point of $x$ with respect to the side
$BC$, $z$ is the symmetric point of $x$ with respect to $AD$, and
$w$ is the centrally symmetric point of $x$ with respect to the
vertex $C$. It is easy to see that $\phi=0$ on the boundary of the
rhombus $\mathfrak{R}$, which implies that $\phi$ can be used as a
trial function for the first Dirichlet eigenvalue
$\lambda_{1}(\mathfrak{R})$. Using a similar argument to that in the
proof of Theorem \ref{theo-4}, particularly the idea of deriving
(\ref{mix0}), one can obtain
 \begin{eqnarray*}
\lambda^{L}_{1}(T_{R})\geq\lambda_{1}(\mathfrak{R}).
 \end{eqnarray*}
Together with Theorem \ref{theo-2}, it yields
 \begin{eqnarray*}
\lambda^{L}_{1}(T_{R})\geq\lambda_{1}(\mathfrak{R})\geq\lambda_{1}(S_{\mathfrak{R}})=\frac{2\pi^{2}}{|S_{\mathfrak{R}}|}=\frac{\pi^{2}}{2|T_{R}|},
\end{eqnarray*}
where $S_{\mathfrak{R}}$ is a square having the same area as
$\mathfrak{R}$, i.e. $|S_{\mathfrak{R}}|=|\mathfrak{R}|$. By
\cite{BC} or \cite[Theorem 1.1]{SB} (i.e. Theorem \ref{theo-3}
here), one knows
\begin{eqnarray*}
\lambda_{1}^{L}(T_{R}^{\ast})=\mu_{2}(T_{R}^{\ast})=\frac{\pi^{2}}{2|T_{R}^{\ast}|}=\frac{\pi^{2}}{2|T_{R}|},
\end{eqnarray*}
which implies the assertion of Theorem \ref{theo-5}.
\end{proof}

Using a similar idea to that shown in \cite{FP}, we have:

\begin{proof}[Proof of Theorem \ref{theo-6}]
As shown in Figure \ref{f1}, for an arbitrary triangle $T:=\triangle
ABC\subset\mathbb{R}^2$, we use red solid line segments to represent
the sides assigned the Dirichlet boundary condition, and the black
dashed line segment to represent the side assigned the Neumann
boundary condition.  Assume that $D$ is the point on the line
segment $BC$ (of $T$) such that $AD$ is perpendicular with $BC$. Let
us enclose the entire triangle $T$ by a rectangular box $GBCH$, as
shown in Figure \ref{f3} below.
\begin{figure}[ht]
\centering
 \includegraphics[width=0.35\textwidth]{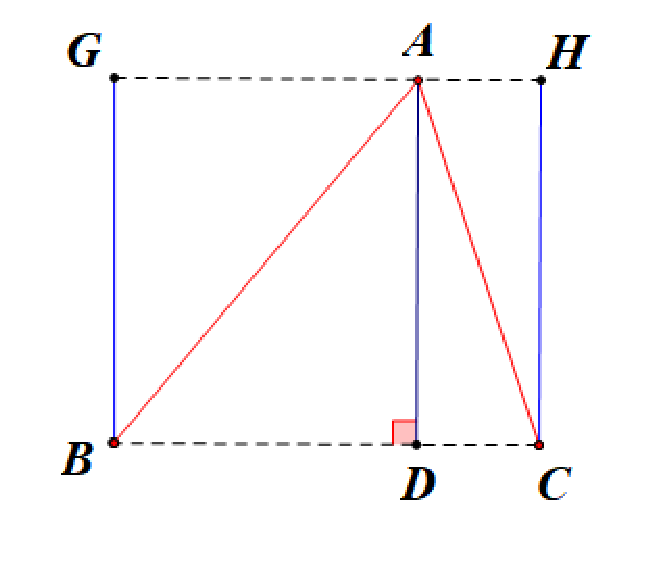}
\caption{The rectangle $GBCH$ constructed to enclose the triangle
$T=\triangle ABC$} \label{f3}
\end{figure}

Let $u$ be a e characteristic function corresponding to the first
mixed eigenvalue $\lambda_{1}^{MS}(T)$. Construct a function $\phi$
defined on the rectangle $GBCH$ as follows:
\begin{eqnarray*}
\phi=\left\{ \begin{array}{lll}
u \qquad&\mathrm{on}~~T,  \\
u_{13} \qquad&\mathrm{on}~~\triangle GAB,\\
u_{24} \qquad&\mathrm{on}~~\triangle AHC,
\end{array}
\right.
\end{eqnarray*}
where $u_{13}$ is the function obtained from $u$ by rotating the
triangle $\triangle ABD$ with respect to the midpoint of its largest
side until it coincides with $\triangle GAB$, and then changing its
sign. Similarly, we can define the function $u_{24}$ from $u$ on
$\triangle ACD$.

Let the length of $AD$ be $h$. For the cylinder $\mathcal{C}$ which
is generated by gluing together two sides $BG$, $CH$ of the
rectangle $GBCH$, consider the Neumann eigenvalue problem of the
Laplacian on $\mathcal{C}$, and then a direct calculation gives us
 \begin{eqnarray} \label{2-4}
\mu_{2}(\mathcal{C})=\min\left\{\frac{\pi^2}{h^2},\frac{4\pi^2}{\ell^2}\right\},
 \end{eqnarray}
where $\mu_{2}(\mathcal{C})$ is the nonzero Neumann eigenvalue of
$\mathcal{C}$. Due to the definition of $\phi$, it is not hard to
see $\int_{\mathcal{C}}\phi=0$, which implies that $\phi$ can be
regarded as a trial function for $\mu_{2}(\mathcal{C})$. Hence,
 we can obtain
\begin{eqnarray*}
\lambda_{1}^{MS}(T)&=& \frac{\int_T|\nabla u|^2}{\int_T u^2}=\frac{\int_{\triangle ABD }|\nabla u|^2 +
\int_{\triangle ACD }|\nabla u|^2}{\int_{\triangle ABD } u^2 + \int_{\triangle ACD } u^2}\\
&=&\frac{\int_{\triangle ABD }|\nabla u|^2+\int_{\triangle ACD
}|\nabla u|^2+\int_{\triangle AGB }|\nabla u_{13}|^2+
\int_{\triangle ACH }|\nabla u_{24}|^2}{\int_{\triangle ABD } u^2+\int_{\triangle ACD } u^2+\int_{\triangle AGB } u_{13}^2+\int_{\triangle ACH } u_{24}^2}\\
&=&\frac{\int_{\mathcal{C}} |\nabla \phi|^2}{\int_{\mathcal{C}}
\phi^2}\\
&\geq&\mu_{2}(\mathcal{C}),
\end{eqnarray*}
which, together with (\ref{2-4}), completes the proof of Theorem
\ref{theo-6}.
\end{proof}

\section{The case of trapezoids} \label{sect3}
\renewcommand{\thesection}{\arabic{section}}
\renewcommand{\theequation}{\thesection.\arabic{equation}}
\setcounter{equation}{0}

We first give the proof of Theorem \ref{theo-8}. However, the
following two facts are needed.

\begin{lemma} \label{lemma3-1}
Let $\Gamma$ be a trapezoid in $\mathbb{R}^2$, and rotate $\Gamma$
around the midpoints of its two parallel sides $m-1$ times. Denote
these $m-1$ copies of $\Gamma$, also including itself, by
$\Gamma=\Omega_{1}$, $\Omega_{2}, \cdots, \Omega_{m}$, and moreover,
supplement the set
$\Omega_{1}\cup\Omega_{2}\cup\cdots\cup\Omega_{m}$ to form a
rectangle $\Omega$, as shown in Figure \ref{f4}. Denote the upper
and lower sides of $\Omega$ by $P^1$ and $P^2$, and the left and
right sides by $Q^1$ and $Q^2$. Then one has
\begin{eqnarray*}
\lambda_{k}^{Q_{1}Q_{2}}(\Gamma)\geq\lambda_{km}^{Q^{1}Q^{2}}(\Omega),\qquad
k=1,2,3,\cdots,
\end{eqnarray*}
 where, as assumed in (\ref{mixed-eigenvalue-1}), $Q_{1}$, $Q_{2}$
 are the left and the right sides of $\Gamma$.
\end{lemma}

\begin{figure}[ht]
\centering
 \includegraphics[width=0.80\textwidth]{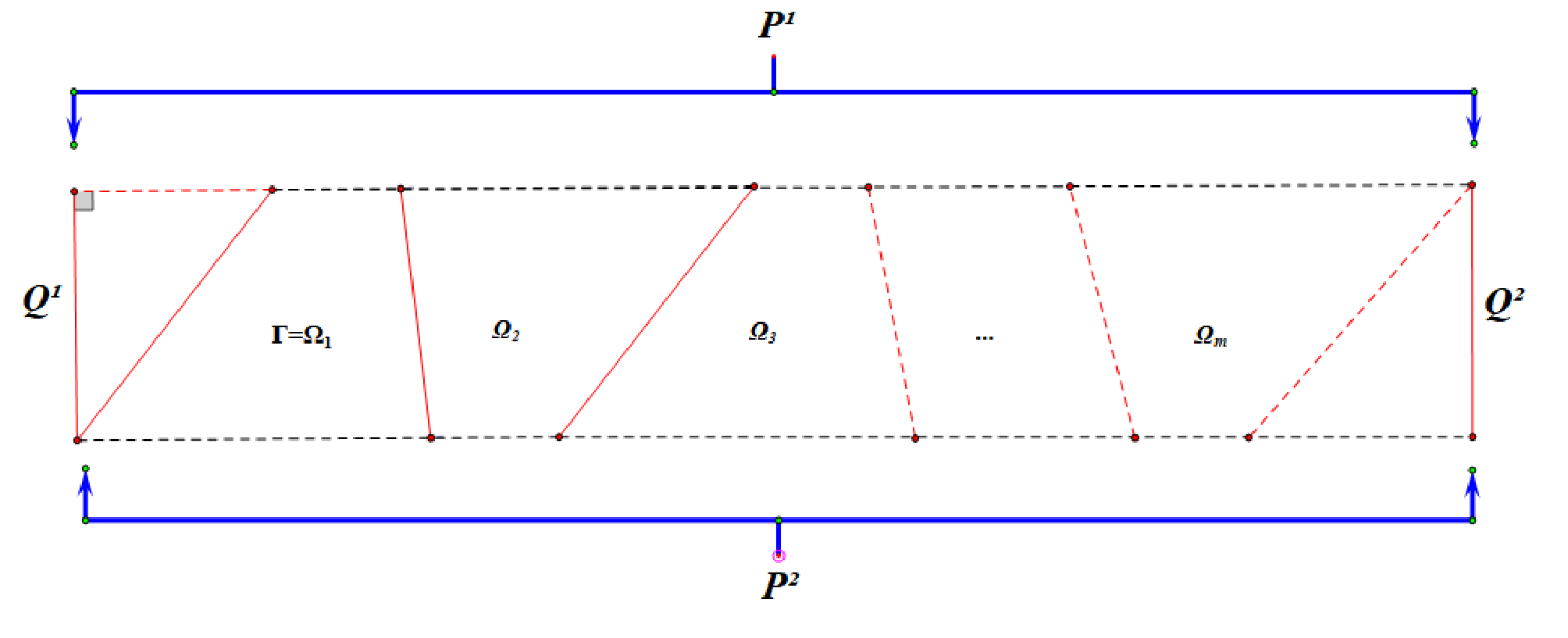}
\caption{A rectangle formed by rotating the trapezoid
$\Gamma\subset\mathbb{R}^2$ and making a supplement} \label{f4}
\end{figure}

\begin{proof}
Assume first that all the eigenvalues of $\Omega_1, \Omega_2,
\cdots, \Omega_m$ are arranged in ascending order as $v_1\leq
v_2\leq v_3\leq\cdots$. Take $\phi_1, \phi_2, \cdots, \phi_{k-1}$ to
be the $k-1$ corresponding eigenfunctions which are mutually
$L^2(\Omega)$ orthogonal. For $j=1,2,\cdots,k$, let $\psi_j:
\Omega\mapsto\mathbb{R}$ be an eigenfunction of $v_j$ when
restricted to the appropriate subdomain, and identically zero
otherwise. Then $\psi_j\in W^{1,2}_{0}(\Omega)$. It is not hard to
see that $\psi_j$ is mutually orthogonal in $L^{2}(\Omega)$. By the
fundamental theory of Linear Algebra, there exists a set of nonzero
numbers $\alpha_1,\alpha_2,\cdots,\alpha_k$ such that
\begin{eqnarray*}
\sum\limits_{j=1}^k\alpha_j\int_\Omega \psi_j\phi_l=0, \qquad
l=1,2,\cdots,k-1.
\end{eqnarray*}
Define a function $f$ as follows
\begin{equation*}
f:=\sum\limits_{j=1}^{k}\alpha_{j}\psi_{j}.
\end{equation*}
It is easy to see that $f$ is $L^2(\Omega)$ orthogonal to
$\phi_1,\phi_2,\cdots,\phi_{k-1}$, which implies that $f$ can be
used as a trial function for $\lambda_{k}^{Q^{1}Q^{2}}(\Omega)$.
Therefore, we have
\begin{eqnarray*}
\lambda_{k}^{Q^{1}Q^{2}}(\Omega)\int_\Omega f^2\leq \int_\Omega
|\nabla f|^2=\sum\limits_{j=1}^{k}\alpha_{j}^{2}v_{j}\int_\Omega
f^{2}\leq v_{k} \int_\Omega f^{2},
\end{eqnarray*}
which directly implies
\begin{eqnarray*}
\lambda_{k}^{Q^{1}Q^{2}}(\Omega)\leq v_{k}.
\end{eqnarray*}
Since each $\Omega_{i}$, $i=1,2,\cdots,m$, is congruent with
$\Gamma$, one knows that $\lambda_{k}^{Q_{1}Q_{2}}(\Gamma)$ will be
repeated $m$ times, and then we can separately arrange all the mixed
eigenvalues of
$\Omega_1\cup\Omega_{2}\cup\Omega_{3}\cup\cdots\cup\Omega_{m}$,
$\Omega$ in the following way (i.e. in the first line, all the mixed
eigenvalues of $\Omega_{i}$, $i=1,2,\cdots,m$, have been listed
non-decreasingly, while in the second line, all the mixed
eigenvalues of the rectangle $\Omega$ of the same type have been
listed non-decreasingly)
\begin{equation*}
\bigg\{ \begin{array}{l}
\lambda_{1}^{Q_{1}Q_{2}}(\Gamma),\lambda_{1}^{Q_{1}Q_{2}}(\Gamma),\cdots,\lambda_{1}^{Q_{1}Q_{2}}
(\Gamma),\lambda_{2}^{Q_{1}Q_{2}}(\Gamma),\lambda_{2}^{Q_{1}Q_{2}}(\Gamma),\cdots,\lambda_{2}^{Q_{1}Q_{2}}(\Gamma),
\lambda_{3}^{Q_{1}Q_{2}}(\Gamma),\cdots\\
\lambda_{1}^{Q^{1}Q^{2}}(\Omega),\lambda_{2}^{Q^{1}Q^{2}}(\Omega),\cdots,\lambda_{m}^{Q^{1}Q^{2}}
(\Omega),\lambda_{m+1}^{Q^{1}Q^{2}}(\Omega),\lambda_{m+2}^{Q^{1}Q^{2}}(\Omega),\cdots,
\lambda_{2m}^{Q^{1}Q^{2}}(\Omega),\lambda_{2m+1}^{Q^{1}Q^{2}}(\Omega),\cdots
\end{array}
\end{equation*}
which implies the assertion of Lemma \ref{lemma3-1} by making a
direct comparison between corresponding eigenvalues of the above two
sequences.
\end{proof}

\begin{lemma}  \label{lemma3-2}
For a rectangle $D$ with length $a$ and width $b$, we still use
$P^1$, $P^2$ to represent the upper and lower sides, and $Q^1$,
$Q^2$ to represent the left and right sides. Then it follows
\begin{equation*}
\lambda_{k}^{Q^{1}Q^{2}}(D)\geq
\frac{\pi^2}{b^2}M\left(\frac{b}{a}k\right),
\end{equation*}
where $M(x)$ is the inverse function of the function
$x=\sum\limits_{j=0}^\infty(y-j^2)_+^{\frac{1}{2}}$, $y\geq0$.
Besides, here an abbreviation
\begin{eqnarray*}
z_{+}:=\frac{z+|z|}{2}= \left\{
 \begin{array}{ll}
z, \quad &when~z\geq0\\
0, \quad &when~z<0
 \end{array}
 \right.
\end{eqnarray*}
has been used.
\end{lemma}

\begin{remark}
\rm{ From \cite[p. 426]{PG}, one would see that $M(x)$ is a strictly
increasing continuous function. }
\end{remark}

\begin{proof}
 Assume that $N$ is the number of
eigenvalues whose values are less or equal to
$\lambda_{k}^{Q^{1}Q^{2}}(D)$. It is not hard to see that $k\leq N$.
Using a method similar with calculating all the Dirichlet
eigenvalues of the Laplacian on rectangles, one can obtain all the
mixed eigenvalues $\lambda_{l}^{Q^{1}Q^{2}}(D)$ of the Laplacian on
the rectangle $D$ as follows
\begin{eqnarray*}
\pi^2\left(\frac{i^2}{a^2}+\frac{j^2}{b^2}\right), \qquad i=1,2,3
,\cdots,~~j=0,1,2,\cdots.
\end{eqnarray*}
Solving the inequality
\begin{eqnarray*}
\pi^2\left(\frac{i^2}{a^2}+\frac{j^2}{b^2}\right)\leq
\lambda_{k}^{Q^{1}Q^{2}}(D),
\end{eqnarray*}
which implies
\begin{eqnarray*}
i\leq
\frac{a}{b}\left(\frac{b^2\lambda_k^{Q^{1}Q^{2}}(D)}{\pi^2}-j^2\right)_{+}^\frac{1}{2}.
\end{eqnarray*}
Hence, we have
\begin{eqnarray*}
k\leq N=\sum\limits_{j=0}^\infty
\left[\frac{a}{b}\left(\frac{b^2\lambda_k^{Q^{1}Q^{2}}(D)}{\pi^2}-j^2\right)_{+}^\frac{1}{2}\right].
\end{eqnarray*}
Then, Using the definition of the function $M(x)$ and its
monotonicity, we further get
\begin{eqnarray*}
\lambda_{k}^{Q^{1}Q^{2}}(D)\geq
\frac{\pi^2}{b^2}M\left(\frac{b}{a}k\right).
\end{eqnarray*}
This completes the proof of Lemma \ref{lemma3-2}.
\end{proof}

Now, we have:

\begin{proof} [Proof of Theorem \ref{theo-8}]
Let $\widetilde{\Gamma}^{\ast}$ be a rectangle that contains $n$
copies of $\Gamma$, with its average width chosen to be $mn + c$ and
its height still $h$, where $c$ depends only on the selection of
$\Gamma$ and is independent of $n$. By Lemma \ref{lemma3-1}, one has
\begin{eqnarray} \label{3-1}
\lambda_{k}^{Q_{1}Q_{2}}(\Gamma)\geq\lambda_{kn}^{Q^{1}Q^{2}}(\widetilde{\Gamma}^{\ast}).
\end{eqnarray}
By Lemma \ref{lemma3-2}, one gets
\begin{eqnarray} \label{3-2}
\lambda_{kn}^{Q^{1}Q^{2}}(\widetilde{\Gamma}^{\ast})\geq\frac{\pi^2}{h^2}M\left(\frac{hkn}{mn+c}\right).
\end{eqnarray}
Combining (\ref{3-1}) and (\ref{3-2}), letting
$n\rightarrow+\infty$, and using the continuity of $M(x)$, we can
obtain
\begin{eqnarray*}
\lambda_{k}^{Q_{1}Q_{2}}(\Gamma)\geq\frac{\pi^2}{h^2}M\left(\frac{hk}{m}\right).
\end{eqnarray*}
By the definition of $M(x)$ and from \cite[pp. 426-427]{PG}, one
knows that when $0\leq x\leq 1$, $M(x)=x^2$. Therefore, if
$\frac{h}{m}\leq\frac{1}{k}$ (i.e. $0<\frac{h}{mk}\leq1$), we have
\begin{eqnarray} \label{3-3}
\lambda_{k}^{Q_{1}Q_{2}}(\Gamma)\geq\frac{\pi^2}{h^2}\frac{k^{2}h^{2}}{m^2}=\frac{\pi^{2}k^{2}}{m^2}.
\end{eqnarray}
Now, we divide the proof into two cases as follows:

Case I. When $k=1$, one knows
 \begin{eqnarray*}
\lambda_{1}^{Q^{1}Q^{2}}(\Gamma^{\ast})=\frac{\pi^2}{m^2},
 \end{eqnarray*}
 with $\Gamma^{\ast}$ the rectangle of width $m$ and height $h$
 described in Theorem \ref{theo-8}, and then from (\ref{3-3}), it
 follows
 \begin{eqnarray*}
\lambda_{1}^{Q_{1}Q_{2}}(\Gamma)\geq\lambda_{1}^{Q^{1}Q^{2}}(\Gamma^{\ast}).
 \end{eqnarray*}

 Case II. When $k\geq2$, it is easy to have
 \begin{eqnarray} \label{3-4}
\frac{h^2}{m^2}\leq\frac{1}{k^2}<\frac{1}{k^{2}-1}.
 \end{eqnarray}
As we know, all the mixed eigenvalues
$\lambda_{l}^{Q^{1}Q^{2}}(\Gamma^{\ast})$ of the rectangle
$\Gamma^{\ast}$ are
\begin{eqnarray*}
\pi^2\left(\frac{i^2}{m^2}+\frac{j^2}{h^2}\right), \qquad i=1,2,3
,\cdots,~~j=0,1,2,\cdots.
\end{eqnarray*}
Together with (\ref{3-4}), we have
\begin{eqnarray*}
\pi^2\left(\frac{i^2}{m^2}+\frac{j^2}{h^2}\right)=\frac{\pi^2}{m^2}\left(i^{2}+j^{2}\cdot\frac{m^2}{h^2}\right)
>\frac{\pi^2}{m^2}\left[i^{2}+j^{2}\cdot(k^{2}-1)\right].
\end{eqnarray*}
Hence, in our setting here, the $k$-th mixed eigenvalue
$\lambda_{k}^{Q^{1}Q^{2}}(\Gamma^{\ast})$ should be attained when
$i=k$, $j=0$, that is,
$\lambda_{k}^{Q^{1}Q^{2}}(\Gamma^{\ast})=\pi^{2}k^{2}/m^{2}$.
Together with (\ref{3-3}), one has
 \begin{eqnarray*}
\lambda_{k}^{Q_{1}Q_{2}}(\Gamma)\geq\lambda_{k}^{Q^{1}Q^{2}}(\Gamma^{\ast}).
 \end{eqnarray*}

In sum, under the assumption that $\frac{h}{m}\leq\frac{1}{k}$, we
always have
$\lambda_{k}^{Q_{1}Q_{2}}(\Gamma)\geq\lambda_{k}^{Q^{1}Q^{2}}(\Gamma^{\ast})$,
which completes the proof of Theorem \ref{theo-8}.
\end{proof}

At the end, we wish to mainly use the idea of proving Theorem
\ref{theo-4} to get our last conclusion.

\begin{proof} [Proof of Theorem \ref{theo-9}]
We fold the right trapezoid $\Gamma_{R}$ along the right angle side
$w_2$ to transform it into an isosceles trapezoid $I_{s}$, such that
$|I_{s}|=2|\Gamma_{R}|$. Let $v$ be the eigenfunction corresponding
to the mixed eigenvalue $\lambda_{1}^{l_{1}l_{2}w_{1}}(\Gamma_R)$,
and construct a function $\rho$ as follows
\begin{eqnarray*}
\rho(x)= \left\{
 \begin{array}{ll}
  v(x), \qquad&\mathrm{if}~ x\in \Gamma_{R},\\
  v(y), \qquad&\mathrm{if}~ x\in
  I_{s}\setminus\Gamma_{R},
  \end{array}
  \right.
\end{eqnarray*}
where $y$ is the symmetric point of $x$ with respect to $w_2$. It is
easy to check that $\rho|_{\partial I_{s}}=0$, which implies that
$\rho$ can be regarded as a trial function for the first Dirichlet
eigenvalue $\lambda_{1}(I_{s})$ of the Laplacian on $I_{s}$. Based
on the symmetry of $\rho$ and using the idea in the proof of Theorem
\ref{theo-4}, we can obtain
\begin{eqnarray*}
\lambda_{1}^{l_{1}l_{2}w_{1}}(\Gamma_R)\geq\lambda_{1}(I_{s}).
\end{eqnarray*}
By Theorem \ref{theo-2}, the above inequality can be improved as
follows
\begin{eqnarray}  \label{3-5}
\lambda_{1}^{l_{1}l_{2}w_{1}}(\Gamma_R)\geq\lambda_{1}(I_{s})\geq\lambda_{1}(S_{I_{s}})=\frac{2\pi^2}{|I_s|}=\frac{\pi^2}{|\Gamma_{R}|},
\end{eqnarray}
where $S_{I_{s}}$ denotes a square with area equal to $|I_{s}|$.
Moreover, the equality symbol in the second inequality of
(\ref{3-5}) can be achieved if and only if the isosceles trapezoid
$I_{s}$ is congruent with the square $S_{I_{s}}$. Since the right
hand side of (\ref{3-5}) can be rewritten as
\begin{eqnarray*}
\pi^{2}\left(\frac{1}{2|\Gamma_{R}|}+\frac{1}{2|\Gamma_{R}|}\right)=\pi^2\left(\frac{1}{2|\Gamma_{R}|}+\frac{1}{4\left(\frac{\sqrt{2|\Gamma_{R}|}}{2}\right)^2}\right),
\end{eqnarray*}
which is exactly the first mixed eigenvalue
$\lambda_{1}^{l_{1}l_{2}w_{1}}(\Gamma_{R}^{\ast})$ of a rectangle
$\Gamma_{R}^{\ast}$ with area $|\Gamma_{R}^{\ast}|=|\Gamma_{R}|$ and
a length-to-width ratio $2:1$ (see e.g. \cite{BC} for this fact), we
further have
\begin{eqnarray} \label{3-6}
\lambda_{1}^{l_{1}l_{2}w_{1}}(\Gamma_R)\geq\lambda_{1}^{l_{1}l_{2}w_{1}}(\Gamma_{R}^{\ast}).
\end{eqnarray}
By the way, one can use a line segment (joining two midpoints of a
pair of parallel sides of the square $S_{I_{s}}$) to divide
$S_{I_{s}}$ into two equal parts, and clearly each part is congruent
with the rectangle $\Gamma_{R}^{\ast}$. Based on the symmetric
property of an eigenfunction of the first Dirichlet Laplacian
eigenvalue $\lambda_{1}(S_{I_{s}})$ with respect to the line
segment, and using an argument similar to the one shown the fact
$\lambda^{MS}_{1}(T^{\ast})=\lambda_{1}(S_{Q})$ in Section
\ref{sect2}, it is not hard to get
$\lambda_{1}(S_{I_{s}})=\lambda_{1}^{l_{1}l_{2}w_{1}}(\Gamma_{R}^{\ast})$.

If the equality in (\ref{3-6}) holds, then from the above argument
one knows $\lambda_{1}(I_{s})=\lambda_{1}(S_{I_{s}})$, which by
Theorem \ref{theo-2} implies that $I_{s}$ is congruent with the
square $S_{I_{s}}$. Since $I_{s}$ is generated by folding
$\Gamma_{R}$ along its right angle side $w_2$, one knows
 in this situation that $\Gamma_{R}$ must be congruent with $\Gamma_{R}^{\ast}$. This
completes the proof of Theorem \ref{theo-9}.
\end{proof}

\section*{Acknowledgments}
\renewcommand{\thesection}{\arabic{section}}
\renewcommand{\theequation}{\thesection.\arabic{equation}}
\setcounter{equation}{0} \setcounter{maintheorem}{0}

This research was supported in part by the NSF of China (Grant Nos.
11801496 and 11926352), the Fok Ying-Tung Education Foundation
(China), the Key Project of Jiangxi Provincial Natural Science
Foundation (Grant No. 20252BAC250003), Hubei Key Laboratory of
Applied Mathematics (Hubei University), and Key Laboratory of
Intelligent Sensing System and Security (Hubei University), Ministry
of Education.

\section*{Conflict of interest}

The authors declare that there are no conflicts of interests
regarding the publication of this paper.

\section*{Data availability statement}

Data sharing is not applicable to this article as no new data were
created or analyzed in this study.


\begin{thebibliography}{9999}

\bibitem{BC} C. Bandle, \emph{Isoperimetric Inequalities and
Applications}, Monographs and Studies in Mathematics, vol. 7, Pitman
(Advanced Publishing Program), Boston-Mass.-London (1980).

\bibitem{BR} R. Brown, \emph{The mixed problem for Laplace's equation in a class of Lipschitz
domains}, Comm. Partial Differential Equations {\bf 19}(7-8) (1994)
1217--1233.

\bibitem{CM} R. F. Chen, J. Mao, \emph{Several isoperimetric inequalities of Dirichlet and Neumann eigenvalues of the Witten
Laplacian}, J. Spectral Theory {\bf 15} (2025) 1241--1277.


\bibitem{DD} D. Daners, \emph{A Faber-Krahn inequality
for Robin problems in any space dimension}, Math. Ann. {\bf 335}
(2006) 767--785.

\bibitem{GF} G. Faber, \emph{Beweis, dass unter allen homogenen membranen von gleicher
fl\"{a}che und gleicher spannung die kreisf\"{o}rmige den tiefsten
Grundton gibt}, Sitz. bayer. Akad. Wiss. (1923), 169--172.

\bibitem{FP} P. Freitas, \emph{Upper and lower bounds for the first Dirichlet eigenvalue of a
triangle}, Proc. Amer. Math. Soc. {\bf 134} (2006) 2083--2089.


\bibitem{AH} A. Henrot, \emph{Minimization problems for eigenvalues of the
Laplacian}, J. Evol. Equ. {\bf 3} (2003) 443--461.

\bibitem{AH1} A. Henrot, \emph{Extremum Problems for Eigenvalues of Elliptic
Operators}, Birkh\"{a}user Verlag, Basel-Boston-Berlin (2006).



\bibitem{EK1} E. Krahn, \emph{\"{U}ber eine von Rayleigh formulierte Minimaleigenschaft des
Kreises}, Math. Ann. {\bf 94} (1925) 97--100.


\bibitem{OB} K. A. Ott, R. Brown, \emph{The mixed problem for the Laplacian in Lipschitz domains}, Potential Anal. {\bf38}(4) (2013)
1333--1364.

\bibitem{RO} R. Osserman, \emph{The isoperimetric inequality}, Bull. Amer. Math. Soc. {\bf 84} (1978) 1182--1238.

\bibitem{PG} G. P\'{o}lya, \emph{On the eigenvalues of vibrating
membranes}, Proc. London Math. Soc. {\bf 3} (1961) 419--433.

\bibitem{PS} G. P\'{o}lya, G. Szeg\"{o}, \emph{Isoperimetric Inequalities in Mathematical Physics}, Ann.
Math. Studies 27, Princeton Univ. Press, New Jersey (1951).



\bibitem{SB} B. Siudeja, \emph{On mixed Dirichlet-Neumann eigenvalues of
triangles}, Proc. Amer. Math. Soc. {\bf 144}(6)  (2016) 2479--2493.


\bibitem{ZS} S. Zaremba, \emph{Sur un probleme toujours possible comprenant, a titre de cas particuliers, le
probleme de Dirichlet et celui de Neumann}, J. Math. Pures Appl.
{\bf 6} (1927) 127--163.


\end{thebibliography}
\end{document}